\documentclass[10pt,onesite,draft]{article}
\newfam\cyrfam

\usepackage{amssymb,latexsym,amscd,amsmath,amsfonts, amsthm,enumerate}

\newtheorem{theorem}{Theorem}[section]
\newtheorem{lemma}[theorem]{Lemma}

\newtheorem{corollary}[theorem]{Corollary}

\newtheoremstyle{definition}% name
  {6pt}%      Space above
  {6pt}%      Space below
  {}%         Body font
  {}%         Indent amount (empty = no indent, \parindent = para indent)
  {\bfseries}% Thm head font
  {.}%        Punctuation after thm head
  {.5em}%     Space after thm head: " " = normal interword space;
  {}%

\theoremstyle{definition}
\newtheorem{definition}[theorem]{Definition}

\newtheoremstyle{remark}% name
  {6pt}%      Space above
  {6pt}%      Space below
  {}%         Body font
  {}%         Indent amount (empty = no indent, \parindent = para indent)
  {\bfseries}% Thm head font
  {.}%        Punctuation after thm head
  {.5em}%     Space after thm head: " " = normal interword space;
  {}%

\theoremstyle{remark}
\newtheorem{remark}[theorem]{Remark}

\newtheoremstyle{note}% name
  {6pt}%      Space above
  {6pt}%      Space below
  {}%         Body font
  {}%         Indent amount (empty = no indent, \parindent = para indent)
  {\bfseries}% Thm head font
  {.}%        Punctuation after thm head
  {.5em}%     Space after thm head: " " = normal interword space;
  {}%

\theoremstyle{note}

\makeatletter
\renewcommand\@makefntext[1]{%
\setlength\parindent{1em}%
\noindent \makebox[1.8em][r]{}{#1}} \makeatother

\begin{document}
\parskip 4pt
\large \setlength{\baselineskip}{15 truept}
\setlength{\oddsidemargin} {0.5in} \overfullrule=0mm
\def\bfh{\vhtimeb}
\date{}
\title{\bf \large TOWARD MIXED MULTIPLICITIES\\ AND JOINT
REDUCTIONS}
\def\b{\vntime}
\author{
 Truong Thi Hong Thanh  and  Duong Quoc Viet   \\
}
 \date{}
\maketitle \centerline{
\parbox[c]{10.5cm}{
\small{\bf ABSTRACT:} In the direction towards the question when mixed multiplicities are equal to the
Hilbert-Samuel multiplicity of joint reductions, this paper not only generalizes \cite[Theorem 3.1]{VDT} that covers the Rees's theorem \cite[Theorem 2.4]{Re}, but also removes the hypothesis that
 joint reductions are systems of parameters in \cite[Theorem 3.1]{VDT}. The results of the paper seem to make the problem of expressing mixed
multiplicities into the Hilbert-Samuel multiplicity of joint
reductions become closer.
 }}

 \section{Introduction} \noindent
Let $(A, \frak{m})$ be a Noetherian local ring  with maximal ideal
$\mathfrak{m}$ and infinite residue field $k = A/\mathfrak{m}.$  Let $M$
be a finitely generated $A$-module.   Let $J$ be an $\frak
m$-primary ideal, $ I_1,\ldots, I_d$ be ideals of $A.$    Put  $I
= I_1\cdots I_d;$ $\overline {M}= M/0_M: I^\infty$ and
\begin{align*}
&{\bf n} =(n_1,\ldots,n_d);{\bf k} = (k_1,\ldots,k_d); {\bf
0}=(0,\ldots,0);
\\
&\mathbf{e}_i = (0, \ldots,  1,  \ldots, 0)\in \mathbb{N}^{d}
(\text{the }i\text{th  coordinate  is  } 1);\\
& \mathrm{\bf I}= I_1,\ldots,I_d;  \mathrm{\bf I}^{[\mathrm{\bf
k}]}= I_1^{[k_1]},
  \ldots,I_d^{[k_d]};  \mathbb{I}^{\mathrm{\bf n}}= I_1^{n_1}\cdots I_d^{n_d};
 \\
&{\bf n^k}= n_1^{k_1}\cdots n_d^{k_d}; \mathbf{k}!= k_1!\cdots
k_d!; \;|{\bf k}| = k_1 + \cdots + k_d.
\end{align*}
  Assume  that $I
\nsubseteq \sqrt{\mathrm{Ann}M},$ then $\overline {M} \not= 0.$
Set  $q=\dim \overline {M}$.
 Then
$\ell\Big(\dfrac{J^{n_0}\mathbb{I}^{\bf
n}M}{J^{n_0+1}\mathbb{I}^{\bf n}M}\Big)$ is a  polynomial of total
degree $q-1$ for all large enough $n_0, \bf n$ by
\cite[Proposition 3.1]{Vi} (see \cite{MV}). Denote by $P$ this
polynomial.  Write the terms of total degree $q-1$  of
$P$ in the form $\sum_{k_0 + |\mathbf{k}| = q - 1}
e(J^{[k_0+1]},\mathrm{\bf I}^{[\mathrm{\bf k}]};
M)\frac{n_0^{k_0}\mathbf{n}^\mathbf{k}}{k_0!\mathbf{k}!},$ then
$e(J^{[k_0+1]},
\mathbf{I}^{[\mathbf{k}]}; M)$
 are non-negative integers not all zero, and
 $e(J^{[k_0+1]},
\mathbf{I}^{[\mathbf{k}]}; M)$ is called   the {\it mixed
multiplicity of $M$ with respect to $J, \bf I$ of the type
$(k_0+1, \bf k)$} (see e.g. \cite{MV, Ve, Vi}).
\footnotetext{\begin{itemize} \item[ ]{\bf Mathematics Subject
Classification (2010):} Primary 13H15. Secondary 13C15, 13D40,
14C17.  \item[ ]{\bf  Key words and phrases:} Mixed multiplicity,
 Hilbert-Samuel multiplicity,  joint reduction.
\end{itemize}}

It has long been
known that the mixed multiplicity of ideals is an important
invariant of algebraic geometry and commutative algebra.
 In past
years, mixed multiplicities of ideals have attracted much
attention (see e.g. [3$-$10, 13$-$34]).

We turn now to some  facts  of mixed multiplicities  $e(J^{[k_0+1]},
\mathbf{I}^{[\mathbf{k}]}; M)$ related to joint reductions.
 In the case of ideals of dimension
$0,$ Risler and Teissier in 1973 \cite{Te} defined mixed
multiplicities  and interpreted them as Hilbert-Samuel
multiplicities of ideals generated by general elements;
 Rees in 1984 \cite{Re}
built joint reductions and showed  that each mixed multiplicity is
the multiplicity of a joint reduction. And under the viewpoint of
joint reductions, Rees's mixed multiplicity theorem can be viewed as an
extension of the result of  Risler and Teissier.
 Developing this theorem of Rees,  B\"{o}ger extended to the case
 of non-$\frak m$-primary ideals (see \cite{SH}); Swanson proved the converse
of Rees's mixed multiplicity theorem for formally equidimensional rings
\cite{Sw} (see \cite[Theorem 17.6.1]{SH});  Dinh-Thanh-Viet
\cite[Theorem 3.1]{VDT} extended to the case that the ideal $I$ have
height larger than $|\mathbf{k}|.$  But whether there is a similar result for arbitrary
ideals, is not yet known.  This is one of motivations to direct  us towards
the following question.

\noindent{\bf Question 1:} When are mixed multiplicities equal to
the Hilbert-Samuel multiplicity of joint reductions?

 Recall that
 the notion of
joint reductions was built in Rees's work in 1984 \cite{Re}. And
O'Carroll in 1987 \cite{Oc} proved the existence of joint
reductions in the general case.  This concept was studied
in \cite{{SH}, Vi2,  Vi4,  VDT, {VT4}}.

\begin{definition}[{Definition \ref{de01}}]
Let $\frak I_i$ be a sequence  consisting $k_i$ elements of $I_i$
for all $1 \le i \le d.$
 Put  ${\bf x} = \frak I_1, \ldots, \frak
 I_d$  and $(\emptyset) = 0_A$. Then ${\bf x}$ is called a {\it joint
reduction} of $\mathbf I$ with respect to $M$ of the type ${\bf
k}=(k_1,\ldots,k_d)$ if $\mathbb{I}^{\mathbf{n} }M =
\sum_{i=1}^d(\frak I_i) \mathbb{I}^{\mathbf{n} - \mathbf{e}_i}M$
for all large $\bf n.$  If $d=1$ then $(\frak I_1)$
is called a {\it reduction} of $I_1$ with respect to $M$
\cite{NR}.
\end{definition}

 The direction towards Question 1 leads us to the following
 result.
 \begin{theorem}[{Theorem \ref{thm2.19vt}}]\label{thm2.31/1} Let ${\bf x}$ be a joint
reduction of ${\bf I}, J$ with respect to $M$ of the type $({\bf k}, k_0+1)$ with  $k_0 + |{\bf k}| = \dim \overline M-1.$
 Assume that $\dim {M}/I{M} < \dim {M} - |\bf k|.$ Then ${\bf x}$ is a system of parameters for ${M}$ and
 $e(J^{[k_0 +1]}, \mathbf{I}^{[\mathbf{k}]}; M) = e(\mathbf{x};
{M}).$\end{theorem}
It should be noted that if one omits the assumption
$\dim {M}/I{M} < \dim {M} - |\bf k|,$
then in general, $e(J^{[k_0 +1]}, \mathbf{I}^{[\mathbf{k}]}; M)$ can not
be the Hilbert-Samuel multiplicity of a joint reduction of ${\bf I}, J$ with respect to $M$ of the type $({\bf k}, k_0+1)$ even when this joint reduction is a system of parameters for $M$ (see Remark \ref{rm35}).

\enlargethispage{1.cm}

Theorem \ref{thm2.31/1} not only replaces the condition on the height of the ideal $I$ in the statement of \cite[Theorem 3.1]{VDT}
by a weaker condition on $I$, but also removes the condition that
the joint reduction ${\bf x}$ is a system of parameters for ${M}$ in \cite[Theorem 3.1]{VDT}.
 Theorem \ref{thm2.31/1} seems to make the
problem of expressing mixed multiplicities into the Hilbert-Samuel
multiplicity of joint reductions become closer.

The paper is divided into three sections. Section 2 is devoted to the discussion of some notions and results  used in the paper.
Section 3 proves the main result. And as interesting consequences of the main result, we get Corollary \ref {co4.0a} that is a stronger result than \cite[Theorem 3.1]{VDT}, and Corollary \ref{thm2.19v} which shows that mixed multiplicities are
the Hilbert-Samuel
multiplicity of weak-(FC)-sequences.

\section{Mixed multiplicities and some relative sequences}
In this section, we recall  notions of weak-(FC)-sequences and
joint reductions, and give some facts on the relationship between
mixed multiplicities and these sequences.

Let $(A, \frak{m})$ be a Noetherian local ring  with maximal ideal
$\mathfrak{m}$ and infinite residue field $k = A/\mathfrak{m}.$  Let $M$
be a finitely generated $A$-module.   Let $J$ be an $\frak
m$-primary ideal, $ I_1,\ldots, I_d$ be ideals of $A.$
 Set   $I
= I_1\cdots I_d;$ $\overline {M}= M/0_M: I^\infty;$ $q=\dim \overline {M}$ and
\begin{align*}
&\mathbf{e}_i = (0, \ldots,  \stackrel{(i)}{1},  \ldots, 0); {\bf n} =(n_1,\ldots,n_d); {\bf k} =(k_1,\ldots,k_d); {\bf
0}=(0,\ldots,0)\in  \mathbb{N}^{d}; \\
& |{\bf k}| = k_1 + \cdots + k_d;  \mathrm{\bf I}= I_1,\ldots,I_d;
 \mathbb{I}^{\mathrm{\bf n}}= I_1^{n_1}\cdots I_d^{n_d}; \mathrm{\bf I}^{[\mathrm{\bf
k}]}= I_1^{[k_1]},
  \ldots,I_d^{[k_d]}.
\end{align*}
 And let $I
\nsubseteq \sqrt{\mathrm{Ann}M}.$
Denote by $P(n_0, {\bf n}, J, \mathbf{I}, M)$ the Hilbert polynomial of the function
$\ell\Big(\dfrac{J^{n_0}\mathbb{I}^{\bf
n}M}{J^{n_0+1}\mathbb{I}^{\bf n}M}\Big).$
Then
 recall that $P(n_0, {\bf n}, J,  \mathbf{I}, M)$
 is a  polynomial of total
degree $q-1$ by \cite[Proposition
3.1]{Vi} (see \cite{MV}).
 If one writes the terms of total degree $q-1$ of
$P(n_0, {\bf n}, J,  \mathbf{I}, M)$ in the form $\sum_{k_0 + |\mathbf{k}| = q - 1}
e(J^{[k_0+1]},\mathrm{\bf I}^{[\mathrm{\bf k}]};
M)\frac{n_0^{k_0}\mathbf{n}^\mathbf{k}}{k_0!\mathbf{k}!},$  then
 $e(J^{[k_0+1]},
\mathbf{I}^{[\mathbf{k}]}; M)$ is called   the {\it mixed
multiplicity of $M$ with respect to $J, \bf I$ of the type
$(k_0+1, \bf k)$} (see e.g. \cite{MV, Ve, Vi}).

 Denote by
$\bigtriangleup^{(h_0,\mathbf{h})}Q(n_0, \bf n)$ the $(h_0,\bf
h)$-difference of the polynomial $Q(n_0,\bf n).$ And throughout the paper, we put
$e(J^{[k_0+1]},
\mathbf{I}^{[\mathbf{k}]}; M) =0 $ if $k_0 + |{\bf k}| > q-1.$

\begin{remark}\label{re2.00} Let $e(J^{[k_0+1]},
\mathbf{I}^{[\mathbf{k}]}; M)$  be  the mixed
multiplicity of $M$ with respect to ideals $J,\mathrm{\bf I}$ of
the type $(k_0+1,\mathrm{\bf k}).$ Then $e(J^{[k_0+1]},
\mathbf{I}^{[\mathbf{k}]}; M) = \bigtriangleup^{(k_0,\mathrm{\bf
k})}P(n_0, {\bf n}, J, \mathbf{I}, M).$
\end{remark}

The concept of joint reductions of $\mathfrak{m}$-primary ideals
was given by Rees \cite{Re} in 1984. This concept was extended to
the set of arbitrary ideals by \cite{Oc,{SH}, Vi2,  Vi4,  VDT,
{VT4}}.

\begin{definition} \label{de01}
Let $\frak I_i$ be a sequence  consisting $k_i$ elements of $I_i$
for all $1 \le i \le d$ and $k_1,\ldots,k_d \ge 0.$
 Put  ${\bf x} = \frak I_1, \ldots, \frak
 I_d$  and $(\emptyset) = 0_A$. Then ${\bf x}$ is called a {\it joint
reduction} of $\bf I$ with respect to $M$ of the type ${\bf k}
=(k_1,\ldots,k_d)$ if $\mathbb{I}^{\mathbf{n} }M =
\sum_{i=1}^d(\frak I_i) \mathbb{I}^{\mathbf{n} - \mathbf{e}_i}M$
for all large $\bf n.$  If $d=1$ then $(\frak I_1)$
is called a {\it reduction} of $I_1$ with respect to $M$
\cite{NR}.
\end{definition}

\enlargethispage{1.cm}

The weak-(FC)-sequence, which was defined in \cite{Vi}, is a kind
of superficial sequences,  and it is proven  to be useful in
several contexts (see e.g. \cite{DMT, DV, DQV, VT, VT4}).

\begin{definition}[\cite{Vi}] \label{de011} Set $I
= I_1\cdots I_d.$   An element $x  \in I_i$ $(1 \leqslant i \leqslant d)$ is called a
{\it weak}-(FC)-{\it element} of $\bf I$ with respect to $M$ if
the following conditions are satisfied:
 \begin{enumerate}[(FC1):]
 \item $x{M}\bigcap {\mathbb I}^{\mathbf{n}}{M}
= x {\mathbb I}^{\mathbf{n}-\mathbf{e}_i}{M}$ for all $n_i \gg 0$
 and all $n_1,\ldots,n_{i-1},n_{i+1},
\ldots, n_d \geq 0$.

\item $x$ is an $I$-filter-regular element with respect to $M,$
i.e.,\;$0_M:x \subseteq 0_M: I^{\infty}.$
 \end{enumerate}

Let $x_1, \ldots, x_t$ be elements of $A$. For any $0\leqslant i
\leqslant t,$ set\;      $M_i = {M}\big/{(x_1, \ldots, x_{i})M}$.
Then $x_1, \ldots, x_t$ is called  a {\it weak}-(FC)-{\it
sequence}   of $\mathbf{I}$ with respect to $M$ if $x_{i + 1}$ is
a weak-(FC)-element  of $\mathbf{I}$ with respect to  $M_i$ for
all $0 \leqslant i \leqslant t - 1$. If a weak-(FC)-sequence
consists of $k_i$ elements of $I_i$ $(1 \le i \le d),$ then it is
called  a weak-(FC)-sequence of the type $(k_1,\ldots,k_d).$
\end{definition}

The following remark recalls some  important properties  of weak-(FC)-sequences.
\begin{remark}\label{note12} Let $\frak I_i$ be a sequence  of  elements of $I_i$
for all $1 \le i \le d$.
Assume that $\frak I_1, \ldots, \frak
 I_d$ is a weak-(FC)-sequence  of $\mathbf{I}$ with respect
to $M$. Then
\begin{equation}\label{vttt1}(\frak I_1, \ldots, \frak
 I_d)M \bigcap \mathbb{I}^{\mathbf{n}}{M}
=  \sum_{i=1}^d(\frak I_i)
\mathbb{I}^{\mathbf{n} - \mathbf{e}_i}M \end{equation} for all large $\bf n$ by \cite[Theorem
 3.4 (i)]{Vi4}. And if
 $x \in I_i$  $(1 \le i \le d)$ is  a weak-{\rm (FC)}-element  of
$\mathbf{I}, J$ with respect to $M,$ then
\begin{equation}\label{vttt2}P(n_0, {\bf n}, J, \mathbf{I}, M/xM)= P(n_0, {\bf n}, J, \mathbf{I}, M)-P(n_0, {\bf n}-\mathbf{e}_i, J, \mathbf{I}, M) \end{equation} by \cite [(3)]{DV} (or the proof of \cite[Proposition 3.3 (i)]{MV}).
\end{remark}

 The role of
weak-(FC) sequences in studying mixed multiplicities is showed by
the following note which will be used as a tool in this paper.
\begin{remark}\label{re2.6aa} Set $J= I_0.$ Let
 $x \in I_i$  $(0 \le i \le d)$ be  a weak-{\rm (FC)}-element  of
$\mathbf{I}, J$ with respect to $M$. Then by (\ref{vttt2}), we get  $\dim M/xM:I^\infty \leqslant \dim M/0_M:I^\infty -1$ and
\begin{enumerate}[{\rm (i)}] \item
$P(n_0, {\bf n}, J, \mathbf{I}, M/xM)=\begin{cases}
\bigtriangleup^{(0,\mathrm{\bf e}_i)}P(n_0, {\bf n}, J,
\mathbf{I}, M) \quad \text{ if }\; 1 \le i \le d\\
\bigtriangleup^{(1,\mathrm{\bf 0})}P(n_0, {\bf n}, J, \mathbf{I},
M) \quad \;\text{ if }\; i = 0.
\end{cases} $
\item   If
    $k_i > 0$ for
certain $0 \le i \le d,$ then
 by Remark \ref{re2.00} and (i) we obtain
$$e(J^{[k_0+1]}, \mathbf{I}^{[\mathbf{k}]}; M)  =\begin{cases}
e(J^{[k_0+1]}, \mathbf{I}^{[\mathrm{\bf k} - \mathbf{e}_i]};
 M\big/xM) \;\text{ if }\; 1 \le i \le d\\
e(J^{[(k_0-1)+1]}, \mathbf{I}^{[\mathrm{\bf k}]}; M\big/xM)
\text{ if } \; i = 0.
\end{cases} $$
\item  For any $\mathbf{k} \in \mathbb{N}^{d}$, by
\cite[Proposition 2.3]{VDT} (see \cite[Remark 1]{Vi}), there
exists a weak-$\mathrm{(FC)}$-sequence of $\mathbf{I}$ with
respect to $M$  of the type $\mathbf{k}.$

\item Let  $\mathbf{y}$ be a weak-(FC)-sequence  of $\mathbf{I}, J$ with respect to
 $M.$ Then by (\ref{vttt1}) in Remark \ref{note12}, it follows that $\mathbf{y}$ is a joint reduction of
$\mathbf{I}, J$ with respect to $M$ if and only if $J^{n_0}\mathbb{I}^{\mathbf{n}}(M/(\mathbf{y})M) = 0$ for all large
 $n_0, \mathbf{n}.$
\item  Let $|{\bf
k}|+k_0 = \dim \overline{M}-1.$  Then $\deg P(n_0, {\bf
n}, J, \mathbf{I}, M) = |{\bf
k}|+k_0 $   by
\cite[Proposition 3.1]{Vi} (see \cite{MV}). So $\bigtriangleup^{(k_0+1, \mathrm{\bf k})}P(n_0, {\bf
n}, J, \mathbf{I}, M) = 0.$
      Now if
${\bf x}$ is a weak-{\rm (FC)}-sequence   of $\mathbf{I}, J$ with
respect to $M$ of the type $(\mathrm{\bf k}, k_0+1),$
    then by (i)  we obtain
   $P(n_0, {\bf n}, J, \mathbf{I},
M/({\bf x})M)=\bigtriangleup^{(k_0+1, \mathrm{\bf k})}P(n_0, {\bf
n}, J, \mathbf{I}, M) = 0.$ Hence by (iv), ${\bf x}$ is a joint
reduction of $\mathbf{I}, J$ with respect to $M$ of the type
$(\mathrm{\bf k}, k_0+1).$
\end{enumerate}
\end{remark}

In order to prove the results of this paper, we need the
following facts.

Recall that an ideal $\frak{a}$ of $A$ is called an {\it ideal of definition} of $M$ if $\ell_A(M/\frak{a}M)<\infty,$ and a sequence ${\bf y} = y_1, \ldots, y_n$ of elements
in $ \frak m$ is called a {\it multiplicity system} of $M$ if
$(\bf y)$ is an ideal of definition of $M$ (see e.g. \cite[Page 192]{BH1}).
Let ${\bf y}$ be a
multiplicity system. Then one defines the {\it multiplicity symbol}
of ${\bf y}$ as follows: if $n=0$, then $\ell_A(M) <\infty$, and
set $e(\mathbf{y};M) = \ell_A(M)$. If $n> 0$, set $$e(\mathbf{y};
M) = e(\mathbf{y}'; M\big/y_1M) - e(\mathbf{y}'; 0_{M}:y_1)$$ (see e.g. \cite[Definition 4.7.3]{BH1}).
 It
is well known that $e(\mathbf{y}; M) \not=0$ if and only if   $\bf
y$ is a system of parameters for $M$, and in this case,
$e(\mathbf{y};  M) = e((\mathbf{y});  M)$ is the Hilbert-Samuel
multiplicity of the ideal $(\bf y)$ with respect to $M$ (see e.g.
\cite{BH1}). And  $I
= I_1\cdots I_d;$ $\overline {M}= M/0_M: I^\infty;$ $q=\dim \overline {M}.$

\begin{lemma}\label{no4.3a} Let ${\bf x}$ be a finite sequence of  elements
in $ \frak m.$  Then we have:
\begin{enumerate}[{\rm (i)}]
\item $(\mathbf{x})$ is an ideal of definition of $I^mM$ for all
large enough  $m$
 if and only if $(\mathbf{x})$ is an ideal of
definition of $\overline {M}$. \item   If $\mathbf{x}$ is a
multiplicity system of $\overline{M},$ then $ e(\mathbf{x}; I^mM)
= e(\mathbf{x}; \overline {M})$ for all large enough $m.$
 \item Let $\bf x$ be a joint
reduction of  $\mathbf{I},J$ with respect to $M$ of the type $(
{\bf 0}, k_0+1).$ Then $(\mathbf{x})$ is an ideal of definition of
$\overline {M}$  and  $e(J^{[k_0+1]}, \mathbf{I}^{[\mathbf{0}]};
M) =e(\mathbf{x}; I^mM) = e(\mathbf{x}; \overline{M})$  for all
large enough $m.$  Moreover, $\dim \overline{M} \le k_0+1,$ and
equality holds if only if $e(J^{[k_0+1]},
\mathbf{I}^{[\mathbf{0}]}; M) \not=0,$ and in this case, ${\bf x}$
is a system of parameters for $ \overline {M}.$
\end{enumerate}
\end{lemma}
\begin{proof}
Because  $\sqrt{\mathrm{Ann}[\overline
{M}\big/(\mathbf{x})\overline {M}}]=
\sqrt{\mathrm{Ann}[I^mM\big/(\mathbf{x})I^mM]}$ for all large
enough  $m$, we have (i). By Artin-Rees Lemma, it implies that  $I^mM \cap (0_M: I^\infty) = 0$ for all large enough $m$. Hence
 $I^m\overline {M} \cong
 I^mM$
 for all large enough $m$. From this it follows that $
e(\mathbf{x}; I^mM) = e(\mathbf{x}; I^m\overline {M})$ for all
large enough $m.$
 Therefore  we get (ii) since
 $\dim (\overline {M}\big/I^m\overline {M}) < \dim
\overline {M}.$ The proof of (iii):  Since $\bf x$ is a joint
reduction of the type $( {\bf 0}, k_0+1),$ $(\bf x)$ is a
reduction of $J$ with respect to $I^mM$ for large enough $m.$ Thus
${\bf x}$ is a multiplicity system  of $I^mM$ for large enough $m$
since $J$ is $\frak m$-primary, and so of $\overline {M}$ by (i).
Hence $\dim \overline{M} \le |{\bf x}|= k_0+1.$  By \cite[Lemma
3.2 (i)]{Vi} and (ii) we get  $e(J^{[k_0+1]},
\mathbf{I}^{[\mathbf{0}]}; M) =e(\mathbf{x}; I^mM) = e(\mathbf{x};
\overline{M})$. And thus $\dim \overline{M} =  k_0+1$ if and only
if $e(J^{[k_0+1]}, \mathbf{I}^{[\mathbf{0}]}; M)\not=0,$ in this
case, ${\bf x}$ is a system of parameters for $ \overline {M}.$
 \end{proof}

 \section{Mixed multiplicities and multiplicities of joint reductions}

This section proves Theorem \ref {thm2.19vt} which is the main theorem of the
paper. And as immediate consequences of this theorem, we obtain Corollary \ref {co4.0a} which is a stronger result than \cite[Theorem 3.1]{VDT}, and
Corollary \ref{thm2.19v} that interpreted mixed multiplicities as Hilbert-Samuel
multiplicities of weak-{\rm (FC)}-sequences.

To prove the main theorem, we need the following lemmas.

\begin{lemma}\label{le2019} Let $x, x_1\in I_1$ and $x_2\in J.$ Let $\dim M/IM < \dim M.$ Assume that $x,x_2$ and $x_1,x_2$ are systems of parameters for ${M}$ and are joint
reductions of ${\bf I}, J$ with respect to $M.$   Then
$$e(x,x_2;M)=e(x_1,x_2;M).$$
\end{lemma}
\begin{proof}  Set $$\Pi = \{\frak p \in \mathrm{Min}{M}\mid \dim A/\frak p = \dim
{M}\}.$$  Let $\mathfrak{p} \in \Pi,$ set $B =
A/\mathfrak{p}.$
Since $x_1, x_2$ and $x, x_2$
are both joint reductions of $\mathbf{I}, J$ with respect to $M$ and are systems of parameters for ${M}$,
 $x_1, x_2$ and $x, x_2$
are both joint reductions of $\mathbf{I}, J$ with respect to $B$, and are systems of parameters for ${B}$ (see e.g \cite[Lemma
17.1.4]{SH}). Hence
there exists $m  \gg 0$ such that $(x_1)$ and $(x)$ are both
reductions of $I_1$ with respect to $J^m(C/x_2C),$ here $ C
=\begin{cases} B \;\;\;\;\;\;\;\;\;\;\qquad\text{ if }
  d =1\\
(I_2 \cdots I_d)^mB \;\text{ if }\;  d >1.\\
\end{cases} $ Since $\dim M/IM < \dim M$, $I =
I_1\cdots I_d \nsubseteq \frak p$ and so $\dim B/C < \dim B.$ Hence
since $x_1, x_2$ and $x, x_2$ are  systems of parameters for $B,$
it can be verified that $x_1, x_2$ and $x, x_2$ are  systems of
parameters for  $C.$ Since $J$ is $\frak m$-primary, $$\dim
(C/x_2C)/J^m(C/x_2C) = 0.$$
From this it follows that $x_1$ and $x$ are
systems of parameters for $J^m(C/x_2C).$ Moreover, since $(x_1)$ and $(x)$
are reductions of $I_1$ with respect to $J^m(C/x_2C),$
 by \cite[Theorem
1]{NR}, we get
  $e(x;J^m(C/x_2C)) =
e(I_1;J^m(C/x_2C)) = e(x_1; J^m(C/x_2C)).$   Therefore we obtain   $e(x;J^m(C/x_2C)) = e(x_1;
J^m(C/x_2C)).$ On the other hand since $\dim
(C/x_2C)/J^m(C/x_2C) = 0,$ we get  $e(x; C/x_2C)= e(x;
J^m(C/x_2C))$ and $e(x_1; C/x_2C)= e(x_1; J^m(C/x_2C))$. So $e(x;
C/x_2C)=e(x_1; C/x_2C). $ Remember that  $x_1, x_2$ is  a system of
parameters for  $M$. Hence $x_2\notin \frak{q}$ for all $\frak{q}\in \Pi$. Consequently
 $x_2$ is not  a  zero divisor of $C$. From this it follows that $e(x; C/x_2C) =
e(x,x_2;C)$ and $$e(x_1; C/x_2C) =e(x_1,x_2;C)$$ by \cite[Page 641, lines 27-28, (D)]{AB}.
Therefore we have $e(x,x_2;C)=e(x_1,x_2;C).$ And since $\dim B/C < \dim B,$
$e(x,x_2;B)= e(x,x_2;C)$ and $e(x_1,x_2;C)=e(x_1,x_2;B).$
Hence we obtain $$e(x,x_2;B)= e(x_1,x_2;B).$$
Thus for any $\mathfrak{p} \in \Pi,$
$$e(x,x_2; A/\mathfrak{p})
= e(x_1,x_2;A/\mathfrak{p}).$$ Consequently, we get
$$\sum_{\mathfrak{p}\in \Pi}
\ell_A(M_{\mathfrak{p}})e(x,x_2; A/\mathfrak{p}) = \sum_{\mathfrak{p}\in
\Pi}\ell_A(M_{\mathfrak{p}}) e(x_1,x_2; A/\mathfrak{p}).$$
On the other hand,  $e(x,x_2; {M})
=\sum_{\mathfrak{p}\in \Pi}\ell_A(M_{\mathfrak{p}})e(x,x_2;
A/\mathfrak{p})$  and
$$e(x_1,x_2; M)
=\sum_{\mathfrak{p}\in
\Pi}\ell_A(M_{\mathfrak{p}})e(x_1,x_2; A/\mathfrak{p})$$ (see e.g. \cite[Theorem 11.2.4]{SH}).
Therefore
$e(x,x_2; M)= e(x_1,x_2;{M}).$
\end{proof}
\begin{lemma}\label{bd32} Let ${\bf x}$ be a joint
reduction of ${\bf I}, J$ with respect to $M$ of the type $({\bf k}, k_0+1)$ with  $k_0 + |{\bf k}| = \dim \overline M-1.$
 Assume that $\dim {M}/I{M} < \dim {M} - |\bf k|.$ Then ${\bf x}$ is a system of parameters for ${M}$.
\end{lemma}
\begin{proof} Set $n = k_0 + |{\bf k}| +1.$
Let ${\bf x}=x_1,\ldots,x_n$ be a joint
reduction of ${\bf I}, J$ with respect to $M$ of the type $({\bf k}, k_0+1)$ with ${\bf x}_{\mathbf{I}} = x_1, \ldots,
x_{|\mathbf{k}| } \subset \mathbf{I}$ and $U = x_{|\mathbf{k}|+1},\ldots,x_n \subset J.$
 Then $(U)$ is a reduction of $J$ with respect to $I^m[M/({\bf x}_{\mathbf{I}})M]$ for all
large enough  $m.$ Since $J$ is $\frak m$-primary, therefore $(U)$ is an ideal of
definition of $I^m[M/({\bf x}_{\mathbf{I}})M]$ for all
large enough  $m.$ Hence $(U)$ is also an ideal of
definition of $M/({\bf x}_{\mathbf{I}})M:I^\infty$ by Lemma \ref{no4.3a} (i).
Since $\dim {M}/I{M} < \dim {M} - |\bf k|$ and $\dim {M}/({\bf x}_{\mathbf{I}}){M} \ge \dim  {M} - |\bf k|,$ it follows that if $\frak p \in \mathrm{Min}[{M}/({\bf x}_{\mathbf{I}}){M}]$ such that $\dim A/\frak p = \dim
{M}/({\bf x}_{\mathbf{I}}){M}$ then $I \nsubseteq \frak p.$ Hence $\frak p \in \mathrm{Min}[{M}/({\bf x}_{\mathbf{I}}){M}:I^\infty].$ So $\dim A/\frak p \le \dim {M}/({\bf x}_{\mathbf{I}}){M}:I^\infty.$ Since $(U)$ is an ideal of
definition of $M/({\bf x}_{\mathbf{I}})M:I^\infty,$ we get
$$\dim {M}/({\bf x}_{\mathbf{I}}){M}:I^\infty \le |U| = k_0+1= \dim \overline{M} - |{\bf k}| \le \dim {M} - |\bf k|.$$ Consequently, we obtain
$$\dim {M} - |{\bf k}| \le \dim
{M}/({\bf x}_{\mathbf{I}}){M}=\dim A/\frak p \le \dim {M}/({\bf x}_{\mathbf{I}}){M}:I^\infty \le \dim {M} - |\bf k|.$$ Thus, \begin{equation}\label{pt-001VT}\begin{aligned}\dim
{M}/({\bf x}_{\mathbf{I}}){M}=\dim {M}/({\bf x}_{\mathbf{I}}){M}:I^\infty =\dim {M} - |\bf k|\end{aligned}\end{equation}
 and
  \begin{equation}\label{pt-001VT0}\begin{aligned}&\{\frak p \in \mathrm{Min}[{M}/({\bf x}_{\mathbf{I}}){M}] \mid \dim A/\frak p = \dim
{M}/({\bf x}_{\mathbf{I}}){M}\}\\ &= \{\frak p \in \mathrm{Min}[{M}/({\bf x}_{\mathbf{I}}){M}:I^\infty] \mid \dim A/\frak p = \dim
{M}/({\bf x}_{\mathbf{I}}){M}:I^\infty\}. \end{aligned}\end{equation}
By (\ref{pt-001VT}),  ${\bf x}_{\mathbf{I}}$ is part of a system of parameters for ${M}$ and $U$ is a system of parameters for ${M}/({\bf x}_{\mathbf{I}}){M}:I^\infty.$ Since $U$ is a system of parameters for ${M}/({\bf x}_{\mathbf{I}}){M}:I^\infty,$ it follows by (\ref{pt-001VT0}) that
$U$ is a system of parameters for ${M}/({\bf x}_{\mathbf{I}}){M}.$
 So ${\bf x}$ is a system of parameters for ${M}.$
\end{proof}

\begin{theorem}\label{thm2.19vt} Let ${\bf x}$ be a joint
reduction of ${\bf I}, J$ with respect to $M$ of the type $({\bf k}, k_0+1)$ with  $k_0 + |{\bf k}| = \dim \overline M-1.$
 Assume that $\dim {M}/I{M} < \dim {M} - |\bf k|.$ Then ${\bf x}$ is a system of parameters for ${M}$ and
 $$e(J^{[k_0 +1]}, \mathbf{I}^{[\mathbf{k}]}; M) = e(\mathbf{x};
{M}).$$\end{theorem}
\begin{proof}
By Lemma \ref{bd32}, ${\bf x}$ is a system of parameters for ${M}$. Recall that $n = k_0 + |{\bf k}| +1$ and
 ${\bf x}=x_1,\ldots,x_n.$
Set $$\Pi = \{\frak p \in \mathrm{Min}{M}\mid \dim A/\frak p = \dim
{M}\}.$$
Note that  since ${\bf x}$ is a joint
reduction of ${\bf I}, J$ with respect to $M, $
${\bf x}$ is also a joint
reduction of ${\bf I}, J$ with respect to $A/\frak p$ for all  $\frak p \in \Pi$ (see e.g \cite[Lemma
17.1.4]{SH}). And since $\dim {M}/I{M} < \dim {M} - |\bf k|,$ it follows that
$\dim A/(\frak p, I) < \dim A/\frak p - |\bf k|$ for all  $\frak p \in \Pi.$ Indeed, assume that $\dim {M}/I{M} = \dim {M}-t,$ then
there exist $a_1,\ldots,a_t \in I$ such that $\dim {M}/(a_1,\ldots,a_t){M} = \dim {M}-t.$ So $a_1,\ldots,a_t$ is part of a system of parameters for ${M}.$ Hence $a_1,\ldots,a_t$ is also part of a system of parameters for $A/\frak p.$ Consequently,
$$\begin{aligned}&\dim A/(\frak p, I) \le \dim A/(\frak p, a_1,\ldots,a_t) = \dim A/\frak p -t = \dim {M}-t \\&= \dim {M}/I{M} < \dim {M} - |{\bf k}|=\dim A/\frak p - |\bf k|.\end{aligned}$$
So \begin{equation}\label{pt-001KV}\begin{aligned}
\dim A/(\frak p,I)A < \dim A/\frak p - |{\bf k}|. \end{aligned} \end{equation}
And   ${\bf x}$
is also a system of parameters for $A/\frak p$ for all  $\frak p \in \Pi$ since ${\bf x}$ is a system of parameters for ${M}.$

\noindent
{\bf Note 1.} Set $J= I_0.$ Assume that
$k_i>0$ and without loss of generality, assume that
$x_1\in I_i$ for $0 \le i \le d.$  We will prove that there exists $x\in I_i$
such that $x$ is a weak-(FC)-element and $x, x_2, \ldots,
x_n$ is a joint reduction of $
\mathbf{I}, J$ with respect to $A/\frak p$ of the type $({\bf k}, k_0+1)$ for
all  $\frak p \in \Pi.$
 Indeed, by \cite[Proposition 17.3.2]{SH},
 there exists a Zariski open subset $V$ of $I_i/\frak
mI_i$ such that if $x \in I_i$ with the image $x + \frak mI_i \in V,$
then $x, x_2, \ldots, x_n$ is a joint reduction of $
\mathbf{I}, J$ with respect to
 $A/\mathfrak{p}$ of the type $({\bf k}, k_0+1)$  for all $\mathfrak{p} \in \Pi.$
And since $x_1 + {\frak m} I_i \in V,$ we have $V \neq \emptyset.$
    On the other
hand, by \cite[Proposition 2.3]{VDT}, there exists a non-empty
Zariski open subset $W$ of $I_i/\frak mI_i$ such that if $x \in
I_i$ with $x + \frak mI_i \in W,$ then $x$ is a
weak-(FC)-element  of $\mathbf{I}, J$ with respect to
$A/\mathfrak{p}$ for all $\mathfrak{p} \in \Pi.$  Since $V, W$ are
non-empty Zariski open, it implies that $V\cap W$ is also
a non-empty Zariski open subset of $I_i/\frak
mI_i$. Hence there exists $x \in I_i$ such that
$x+ \frak mI_i \in V\cap W,$ then $x$ is a weak-(FC)-element and $x, x_2, \ldots, x_n$ is a joint reduction of $
\mathbf{I}, J$ with respect to
 $A/\mathfrak{p}$ of the type $({\bf k}, k_0+1)$  for all $\mathfrak{p} \in \Pi.$

 Now, we prove  by induction
 on $k_0+|{\bf k}|$ that
 \begin{equation}\label{eqtv1-2019}e(J^{[k_0+1]},\mathrm{\bf I}^{[\mathrm{\bf k}]}; M) =
e(\mathbf{x};{M}). \end{equation}

If $|{\bf k}| = 0$, then since $\bf x$ is a joint reduction of $
\mathbf{I}, J$ with respect to $M$ of the type $({\bf 0}, k_0+1)$,
by Lemma \ref{no4.3a} (iii) we get
  $e(J^{[k_0+1]},
\mathbf{I}^{[\mathbf{0}]}; M) = e(\mathbf{x};{M}/0_M:I^\infty).$ Since $I \nsubseteq \frak p$ for any $\mathfrak{p} \in \Pi,$ it follows that
$({M}/0_M:I^\infty)_{\frak p} = M_{\frak p}.$ Hence by (\ref{pt-001VT0}), we get
$e(\mathbf{x};{M}/0_M:I^\infty)= e(\mathbf{x};{M})$ (see e.g. \cite[Theorem 11.2.4]{SH}).
So $$e(J^{[k_0+1]},
\mathbf{I}^{[\mathbf{0}]}; M) = e(\mathbf{x};{M}).$$
 Hence we also get the proof of (\ref{eqtv1-2019}) in the  case $k_0+|\mathbf{k}| = 0.$

 Consider the case that $|\mathbf{k}| >0.$ And without loss of generality, assume that
$k_1>0$ and $x_1\in I_1.$ Recall that $e(J^{[k_0+1]},\mathrm{\bf I}^{[\mathrm{\bf k}]}; M)
=\sum_{\mathfrak{p}\in
\Pi}\ell_A(M_{\mathfrak{p}})e(J^{[k_0+1]},\mathrm{\bf
I}^{[\mathrm{\bf k}]}; A/\mathfrak{p})$ by \cite[Theorem 3.2]{VT1}.
Hence to have (\ref{eqtv1-2019}), we need to prove that
\begin{equation}\label{eqtv111} e(J^{[k_0+1]},\mathrm{\bf I}^{[\mathrm{\bf k}]}; A/\mathfrak{p}) =
e({\bf x}; A/\mathfrak{p})\end{equation}
 for any $\mathfrak{p} \in \Pi.$ Let $B= A/\mathfrak{p},$ $\mathfrak{p} \in \Pi.$
By Note 1, there exists $x \in I_1$ such that
$x$ is a weak-(FC)-element and $x, x_2,\ldots,x_n$ is a joint reduction of $
\mathbf{I}, J$ with respect to  $B.$
Since $x$ is a weak-(FC)-element of
$\mathbf{I}, J$ with respect to $B$, by  Remark \ref{re2.6aa}
(ii), $$e(J^{[k_0+1]},
\mathbf{I}^{[\mathbf{k}]}; B)=e(J^{[k_0+1]},
\mathbf{I}^{[\mathbf{k-e}_1]}; B/xB).$$  By (\ref{pt-001KV}), $\dim B/IB < \dim B - |{\bf k}|.$ From this it follows that
$$\begin{aligned}&\dim (B/xB)/I(B/xB) = \dim B/(x,I)B \le \dim B/IB \\&< \dim B - |{\bf k}|= \dim B/xB - |{\bf k}-{\bf e}_1|.\end{aligned}$$
Consequently, \begin{equation}\label{eqtv2019}\dim (B/xB)/I(B/xB) < \dim B/xB - |{\bf k}-{\bf e}_1|. \end{equation}
Note that since $x, x_2,\ldots,x_n$ is a joint reduction of $
\mathbf{I}, J$ with respect to  $B$ of the type $(\mathbf{k}, k_0+1),$ $x_2,\ldots,x_n$ is a joint reduction of $
\mathbf{I}, J$ with respect to  $B/xB$ of the type $(\mathbf{k-e}_1, k_0+1).$ Hence by the
inductive hypothesis, $$e(J^{[k_0+1]},
\mathbf{I}^{[\mathbf{k-e}_1]}; B/xB) = e(x_2,\ldots,x_n;B/xB).$$
Since $x$ is not  a  zero divisor of $B,$ $e(x_2,\ldots,x_n; B/xB) = e(x,x_2,\ldots,x_n; B)$
by \cite[Page 641, lines 27-28, (D)]{AB}. Consequently,  \begin{equation}\label{eqtv20}e(J^{[k_0+1]},\mathbf{I}^{[\mathbf{k}]}; B)=e(x,x_2,\ldots,x_n; B).\end{equation}
So to prove (\ref{eqtv111}), we need to show that
 \begin{equation}\label{eqtv21}e(x_1,x_2,\ldots,x_n; B)=e(x,x_2,\ldots,x_n; B)\end{equation}
via the following cases.

{\bf Case 1:} $k_0+|\mathbf{k}| =1.$

Then $k_0 =0.$ Hence since $\dim B/IB < \dim B - |{\bf k}| < \dim B$ by (\ref{pt-001KV}), we get the proof of (\ref{eqtv21}) by Lemma \ref{le2019}.

{\bf Case 2:} $k_0+|{\bf k}| > 1$ and $|{\bf k}| = 1$ and $\dim B/IB > 1.$

In this case, $k_0 > 0$ and $n \ge 3,$ $x_1 \in I_1$ and $x_2,\ldots, x_n \in J.$
Set $U=x_2,\ldots, x_n.$ Then $\dim B/(U,I)B \le \dim B/(U)B = 1 < \dim B/IB.$
By \cite[Proposition 17.3.2]{SH}, $x_2,\ldots,x_n \in J\setminus{\frak m}J.$
Denote by $x_2',\ldots, x_n'$ the images of $x_2,\ldots, x_n$ in $J/{\frak m}J,$ respectively.
Since $x_2,\ldots, x_n$ is part of the system of parameters ${\bf x},$ it follows that $x_2',\ldots, x_n'$ is part of a basis of $k$-vector space $J/{\frak m}J.$ So $(U)\cap {\frak m}J = {\frak m}(U)$ (see e.g. \cite[Proposition 8.3.3(1)]{SH}).
Hence $((U) + {\frak m}J)/{\frak m}J \cong (U)/{\frak m}(U).$ Consequently,  we can consider $(U)/{\frak m}(U)$ as a $k$-vector subspace of $J/\frak m
J.$  Now, assume  that $P_1,P_2,\ldots, P_t$ are all the prime ideals of $\mathrm{Min}[B/IB]$ such that
$$\dim B/IB  = \mathrm{Coht} P_j\quad (1 \le j \le t).$$  For each $j = 1, \ldots, t$, let $W_j$ be the image of $P_j \bigcap (U)$ in $J/\frak mJ$.
Since $$\dim B/IB > \dim B/(U,I)B,$$ $(U)\setminus \bigcup_{j=1}^tP_j \not=  \emptyset$ by Prime Avoidance.
Hence $W_1, \ldots, W_t$ are proper $k$-vector subspaces of $(U)/\frak m(U)$  by Nakayama's lemma.
Since $k$ is an infinite field,
$$\Lambda = [(U)/\frak m(U)] \setminus \bigcup_{j=1}^tW_j$$
is a non-empty Zariski open subset of $(U)/\frak m(U)$. By \cite[Proposition 17.3.2]{SH},
 there exists a Zariski open subset $\Gamma$ of $J/\frak
mJ$ such that if $x \in J$ with the image $x' \in \Gamma,$
then $x_1, x, \ldots, x_n$ is a joint reduction of $
\mathbf{I}, J$ with respect to
 $B$ of the type $({\bf k}, k_0+1).$ Remember that $x'_2 \in \Gamma \cap [(U)/\frak m(U)]$ since $x_1, x_2, \ldots, x_n$ is a joint reduction.
  So $$\Gamma'=\Gamma \cap [(U)/\frak m(U)]    \not=  \emptyset.$$ Hence
    $\Gamma' \cap \Lambda \not=  \emptyset.$
Now, let $u \in (U)$ such that  the image $u'\in \Gamma' \cap \Lambda.$ Then we have
 $$\dim B/(u,I)B = \dim B/IB -1$$ and $x_1, u, \ldots, x_n$ is a joint reduction of $
\mathbf{I}, J$ with respect to
 $B.$ In this case, $x_1, u, \ldots, x_n$  is a system of parameters for $B.$ Hence
  it can be verified that $u',x_3',\ldots,x_n'$ is also a basis of $k$-vector space $(U)/{\frak m}(U).$ Consequently, $(U)= (u,x_3,\ldots,x_n)$ by Nakayama's lemma.
 So $$(x_1, x_2,\ldots, x_n)= (x_1,u,x_3,\ldots,x_n) \; \mathrm{and}\; (x, x_2,\ldots, x_n)= (x,u,x_3,\ldots,x_n).$$
 Moreover in this case, $\dim B/IB < \dim B -|{\bf k}|$ by (\ref{pt-001KV}), and hence
     $$\begin{aligned} &\dim (B/uB)/I(B/uB)=\dim B/(u, I)B = \dim B/IB -1\\ &< \dim B -|{\bf k}| -1 = \dim B/uB -|{\bf k}|.\end{aligned}$$  Note that $x,x_3,\ldots,x_{n}$ and $x_1,x_3,\ldots,x_{n}$ are joint reductions with respect to $B/uB$ of the type $({\bf k}, k_0).$ Hence by the
inductive hypothesis on $k_0 + |{\bf k}|$, we have $$e(x,x_3,\ldots, x_n;B/uB)=e(J^{[(k_0-1)+1]},\mathrm{\bf I}^{[\mathrm{\bf k}]};
B/uB)= e(x_1,x_3,\ldots, x_n;B/uB).$$ On the other hand, since $u$ is not  a  zero divisor of $B$, by \cite[Page 641, lines 27-28, (D)]{AB}, we get
$$e(x,u,x_3,\ldots, x_n;B) = e(x_1,u,x_3,\ldots, x_n;B).$$ Recall that $(x_1, x_2,\ldots, x_n)= (x_1,u,x_3,\ldots,x_n), (x, x_2,\ldots, x_n)= (x,u,x_3,\ldots,x_n).$
Consequently, $$e(x,x_2,x_3,\ldots, x_n;B) = e(x_1,x_2,x_3,\ldots, x_n;B).$$
Therefore, we obtain the proof of (\ref{eqtv21}) in this case.

{\bf Case 3:} $k_0+|{\bf k}| > 1$ and $|{\bf k}| = 1$ and $\dim B/IB \le 1.$

In this case, $k_0 > 0.$ Consider the case $k_0 = 1.$ Since  $x, x_2$ and $x_1, x_2$ are both joint reductions of $\mathbf{I}, J$ with respect to $B/x_3B$, and are systems of parameters for ${B/x_3B},$ and on the other hand, $$\dim (B/x_3B)/I(B/x_3B)=\dim B/(x_3, I)B \leq \dim B/IB\leq 1< 2= \dim B/x_3B,$$ by Lemma \ref{le2019}, we get $e(x,x_2; B/x_3B) = e(x_1, x_2; B/x_3B)$.
Remember that  $\bf x$ is  a system of
parameters for  $M$. Hence $x_3\notin \frak{p}$. Consequently
 $x_3$ is not  a  zero divisor of $B$.
 So we have $e(x, x_2,x_3;B)= e(x_1, x_2,x_3;B)$   by \cite[Page 641, lines 27-28, (D)]{AB}.

If $k_0 > 1,$ then $\dim B = n \ge 4.$
Then $x_2, x_3, x_4, \dots, x_n \in J$ and we have
$$\dim (B/x_2B)/I(B/x_2B)=\dim B/(x_2, I)B \le \dim B/IB \le 1 < \dim B/x_2B -|{\bf k}|.$$
Consequently, by the
inductive hypothesis on $k_0 + |{\bf k}|$, $$e(x, x_3,\ldots, x_n;B/x_2B)=e(J^{[(k_0-1)+1]},\mathrm{\bf I}^{[\mathrm{\bf k}]};
B/x_2B)= e(x_1,x_3,\ldots, x_n;B/x_2B).$$ Now, recall that $\bf x$ is  a system of
parameters for  $M$. Hence $x_2\notin \frak{p}$. So $x_2$ is not  a  zero divisor of $B.$ Therefore,
$$e(x,x_2,x_3,\ldots, x_n;B) = e(x_1,x_2,x_3,\ldots, x_n;B).$$ We get (\ref{eqtv21}).

{\bf Case 4:} $k_0+|{\bf k}| > 1$ and $|{\bf k}| > 1.$

Then $x_2 \in I_i$ for some $1 \le i \le d.$ Hence by (\ref{eqtv2019}), $$\dim (B/x_2B)/I(B/x_2B) < \dim B/x_2B - |{\bf k}-{\bf e}_i|.$$
Therefore, by the
inductive hypothesis on $k_0 + |{\bf k}|$, $$e(x, x_3,\ldots, x_n;B/x_2B)=e(J^{[k_0+1]},\mathrm{\bf I}^{[\mathrm{\bf k-e}_i]};
B/x_2B)= e(x_1,x_3,\ldots, x_n;B/x_2B).$$ Hence since $x_2$ is not  a  zero divisor of $B$,
$$e(x,x_2,x_3,\ldots, x_n;B) = e(x_1,x_2,x_3,\ldots, x_n;B).$$ We have (\ref{eqtv21}).

From the above cases, we get the proof of (\ref{eqtv21}). Consequently, by $(\ref{eqtv20}),$ we get (\ref {eqtv111}), that
$$e(J^{[k_0+1]},\mathrm{\bf I}^{[\mathrm{\bf k}]}; A/\mathfrak{p}) =
e({\bf x}; A/\mathfrak{p})$$ for all
 $\mathfrak{p} \in \Pi.$
Hence
$$\sum_{\mathfrak{p}\in \Pi}
\ell_A(M_{\mathfrak{p}})e(J^{[k_0+1]},\mathrm{\bf I}^{[\mathrm{\bf
k}]}; A/\mathfrak{p}) = \sum_{\mathfrak{p}\in
\Pi}\ell_A(M_{\mathfrak{p}}) e(\mathbf{x}; A/\mathfrak{p}).$$
 On the other hand, $e(\mathbf{x}; {M})
=\sum_{\mathfrak{p}\in \Pi}\ell_A(M_{\mathfrak{p}})e(\mathbf{x};
A/\mathfrak{p})$ (see e.g. \cite[Theorem 11.2.4]{SH}) and
$$e(J^{[k_0+1]},\mathrm{\bf I}^{[\mathrm{\bf k}]}; M)
=\sum_{\mathfrak{p}\in
\Pi}\ell_A(M_{\mathfrak{p}})e(J^{[k_0+1]},\mathrm{\bf
I}^{[\mathrm{\bf k}]}; A/\mathfrak{p})$$ by \cite[Theorem 3.2]{VT1}.
Therefore, we get (\ref{eqtv1-2019}) in the case that $|{\bf k}| > 0,$ and hence the proof of (\ref{eqtv1-2019}) is complete, that
$e(J^{[k_0+1]},\mathrm{\bf I}^{[\mathrm{\bf k}]}; M)= e(\mathbf{x};
{M}).$
\end{proof}

\begin{remark}\label{rm35}
From Theorem \ref{thm2.19vt}, one may raise a question:
 Does this theorem hold if $\dim {M}/I{M} \ge \dim {M} - |\bf k|$?
 Consider the case $M=A$ and $\dim A =2.$ Let $x_1, x_2$ be a system of parameters for $A.$ Set $\frak I = x_1A$ and $J = (x_1, x_2)A.$ Then $J$ is an $\frak{m}$-primary ideal of $A$ and
$x_1, x_2$ is a joint reduction  of
 $\frak I, J$ with respect to $A$ of the type $(1, 1)$ since $\frak I^nJ^m = x_1\frak I^{n-1}J^m + x_2\frak I^nJ^{m-1}$ for all $n, m \ge 1.$
 In this case, ${\bf k} = (1),$ $k_0 =0$ and $\dim A/\frak IA = 1 = \dim A - |\bf k|.$
  It can be verified that $(x_1) \cap \frak I^{n}J^m = x_1\frak I^{n-1}J^m$ for all $n \ge 1; m \ge 0$ and $0_A: x_1 \subset 0_A: \frak I^\infty.$
  Hence  $x_1$ is a weak-(FC)-element
 of $\frak{I}, J$ with respect to $A.$ Consequently, by Remark \ref{re2.6aa} (ii), we get $e(J^{[1]}, \frak I^{[1]}; A)= e(J^{[1]}, \frak I^{[0]}; A/x_1A).$ Now, since $\frak I = x_1A,$ $\frak I[A/x_1A]= 0.$ So
 we obtain $e(J^{[1]}, \frak I^{[0]}; A/x_1A)=0$.
 Thus
 $e(J^{[1]}, \frak I^{[1]}; A)=0.$ Hence $$e(J^{[1]}, \frak I^{[1]}; A) \not= e(x_1, x_2; A).$$ This shows that
  Theorem \ref{thm2.19vt} does not hold in general if one omits  the assumption $\dim {M}/I{M} < \dim {M} - |\bf k|.$
\end{remark}

Now, we return \cite[Theorem 3.1]{VDT} that covers Rees's theorem \cite[Theorem\ 2.4 (i), (ii)]{Re}  by
the following stronger result than \cite[Theorem 3.1]{VDT}.

\begin{corollary}\label{co4.0a} Let
 ${\bf x}$ be  a joint
reduction of  $\mathbf{I}, J$ with respect to $M$ of the type
$({\bf k},k_0+1)$ with  $k_0 + |{\bf k}| = \dim \overline M-1.$  Assume that
$\mathrm{ht}\Big(\dfrac{I+\mathrm{Ann}M}{\mathrm{Ann}M}\Big) >
|\bf k|.$ Then ${\bf x}$ is a system of parameters for ${M}$ and
 $$e(J^{[k_0 +1]}, \mathbf{I}^{[\mathbf{k}]}; M) = e(\mathbf{x};
{M}).$$
\end{corollary}
\begin{proof} Since $\mathrm{ht}\Big(\dfrac{I+\mathrm{Ann}M}{\mathrm{Ann}M}\Big) >
|\bf k|,$ we have
 $\dim M/IM < \dim M - |{\bf k|}$. Hence from Theorem \ref{thm2.19vt} we get the proof.
\end{proof}

Finally, we would like to give the following result which shows that mixed multiplicities are
the Hilbert-Samuel
multiplicity of weak-(FC)-sequences.

\begin{corollary}\label{thm2.19v} Let ${\bf x}$
 be a weak-$\mathrm{(FC)}$-sequence of $\mathbf{I},
J$ with respect to $M$ of the type $({\bf
k},k_0+1)$ with  $k_0 + |{\bf k}| = \dim \overline M-1.$  Assume that $\dim {M}/I{M} < \dim {M} - |\bf k|.$ Then ${\bf x}$ is a system of parameters for ${M}$ and
 $e(J^{[k_0 +1]}, \mathbf{I}^{[\mathbf{k}]}; M) = e(\mathbf{x};
{M}).$\end{corollary}
\begin{proof} By  Remark
\ref{re2.6aa} (v), ${\bf x}$ is a joint reduction  of $\mathbf{I},
J$ with respect to $M$ of the type $({\bf k},k_0+1).$ Hence the corollary is
shown by Theorem \ref{thm2.19vt}. \end{proof}

\noindent {\bf Acknowledgement:} {\small Published: 26 October 2019  in  Bulletin of the Brazilian Mathematical Society, New Series. This research is funded by Vietnam National Foundation for Science and Technology Development (NAFOSTED) under grant number 101.04.2015.01.}

%----------------------------------------------------------------------------------------------------

\noindent
Department of Mathematics, Hanoi National University of Education\\
136 Xuan Thuy Street, Hanoi, Vietnam\\
 Emails: thanhtth@hnue.edu.vn \;and\;
vduong99@gmail.com


\begin{thebibliography}{99}{\small
\bibitem{AB} M. Auslander,  D. A. Buchsbaum, {\it Codimension and multiplicity}, Ann. Math. 68(1958), 625-657.
\bibitem{BH1} W. Bruns, J. Herzog, {\it Cohen-Macaulay rings}, Cambridge Studies in Advanced  Mathematics, 39, Cambridge,
 Cambridge University Press, 1993.
\bibitem{CP} R. Callejas-Bedregal,  V. H. Jorge P�erez, {\it  Mixed multiplicities and the minimal number of generator of modules}, J. Pure Appl.
Algebra 214 (2010), 1642-1653.

 \bibitem{DMT} L. V. Dinh, N. T. Manh, T. T. H. Thanh, {\it On
 some superficial sequences,}
Southeast Asian Bull. Math. 38 (2014), 803-811.
\bibitem{DV1} L. V. Dinh, D. Q. Viet, {\it On two results of mixed multiplicities}, Int. J. Algebra 4(1) 2010, 19-23.
\bibitem{DV} L. V. Dinh,  D. Q. Viet,
{\it On mixed multiplicities of good filtrations}, Algebra Colloq.
22, 421 (2015) 421-436.



\bibitem{HHRT} M. Herrmann, E.  Hyry,  J.
Ribbe, Z. Tang,  {\it Reduction numbers and multiplicities of
multigraded structures},  J. Algebra 197 (1997), 311-341.
%\bibitem{Hy} E. Hyry, {\it The diagonal subring and the Cohen-Macaulay property of a multigraded ring}, Trans. Amer. Math. Soc. 351(1999), 2213-2232.
\bibitem{KR1} D. Kirby,  D. Rees, {\it Multiplicities in graded rings I: the general theory},
Contemporary Mathematics 159 (1994), 209-267.

\bibitem{KV} D. Katz,  J. K. Verma, {\it Extended Rees algebras and mixed multiplicities},
Math. Z. 202 (1989), 111-128.

\bibitem{MV} N. T. Manh,  D. Q. Viet, {\it Mixed  multiplicities of modules over Noetherian
local rings}, Tokyo J. Math. 29 (2006), 325-345.
 %\bibitem{MV2} N. T.  Manh,  D. Q. Viet, {\it On the Mixed Multiplicities of Multi-graded Fiber Cones}, Tokyo J. Math. 31 (2008), no. 2, 399-414.
\bibitem{NR} D. G.  Northcott, D. Rees, {\it Reduction of ideals in local rings},  Proc. Cambridge Phil.
Soc. 50(1954), 145-158.

  \bibitem{Oc} L. O'Carroll, {\it On two theorems concerning reductions in local rings}, J. Math. Kyoto Univ. 27-1(1987), 61-67.
\bibitem{Re} D. Rees, {\it Generalizations of reductions and mixed multiplicities}, J. London.
Math. Soc. 29 (1984), 397-414.

\bibitem{Sw} I. Swanson, {\it Mixed multiplicities, joint reductions and quasi-unmixed
local rings }, J. London Math. Soc. 48 (1993), no. 1, 1-14.
\bibitem{SH} C. Huneke, I. Swanson, {\it Integral Closure of Ideals, Rings, and Modules}, London Mathematical Lecture Note Series 336, Cambridge University Press (2006).
\bibitem{Te} B. Teissier, {\it Cycles \'evanescents, sections planes,
et conditions de Whitney}, Singularities \`a  Carg\`ese, 1972.
Ast\`erisque, 7-8 (1973), 285-362.
\bibitem{htv} T.  T.  H.  Thanh, D.  Q. Viet, {\it Mixed multiplicities of maximal degrees,} J. Korean Math. Soc. 55 (2018), No. 3, 605-622.
\bibitem{TV2} T. T. H. Thanh,  D. Q. Viet, {\it Mixed multiplicities and the multiplicity of Rees modules of reductions,} J. Algebra Appl. Vol. 18, No. 9 (2019) 1950176 (13 pages).
\bibitem{Tr2} N. V. Trung, {\it Positivity of mixed multiplicities}, J. Math. Ann. 319(2001), 33 - 63.
\bibitem{TV} N. V. Trung,   J. Verma, {\it   Mixed  multiplicities of ideals versus mixed
volumes of polytopes}, Trans. Amer. Math. Soc. 359 (2007),
4711-4727.
\bibitem{Ve} J. K. Verma, {\it Multigraded  Rees algebras and mixed multiplicities}, J. Pure
and  Appl.  Algebra 77 (1992), 219-228.
\bibitem{Vi} D. Q. Viet, {\it Mixed multiplicities of arbitrary ideals in local rings},
Comm. Algebra. 28(8) (2000), 3803-3821.
\bibitem{Vi1} D. Q. Viet, {\it On some properties of $(FC)$-sequences of ideals in local rings},
Proc. Amer. Math. Soc. 131 (2003), 45-53.
\bibitem{Vi2} D. Q. Viet, {\it Sequences determining mixed multiplicities and reductions
of ideals}, Comm. Algebra. 31 (2003), 5047-5069.
%\bibitem{Vi3} D. Q. Viet, {\it Reductions and mixed multiplicities of ideals}, Comm. Algebra. 32(2004), 4159-4178.
\bibitem{Vi4} D. Q. Viet, {\it Reductions and mixed multiplicities of ideals},
Comm. Algebra. 32 (2004), 4159-4178.
\bibitem{Vi5} D. Q. Viet, {\it The multiplicity and the Cohen-Macaulayness of extended Rees
algebras of equimultiple ideals}, J. Pure and Appl. Algebra 205
(2006), 498-509.
\bibitem{DQV} D. Q. Viet, {\it On the Cohen-Macaulayness of fiber cones},  Proc. Amer. Math. Soc. 136 (2008),  4185-4195.
\bibitem{VDT} D. Q. Viet, L. V. Dinh, T. T. H. Thanh,
{\it A note on joint reductions and mixed multiplicities}, Proc.
Amer. Math. Soc. 142 (2014), 1861-1873.
 \bibitem{Vi6} D. Q. Viet,  N. T. Manh,  {\it Mixed multiplicities of multigraded modules},
Forum Math. 25 (2013), 337-361.
\bibitem{VT0} D. Q. Viet, T. T. H. Thanh, {\it Multiplicity and Cohen-Macaulayness of fiber cones
of good filtrations}, Kyushu J. Math. 65(2011), 1-13.
\bibitem{VT} D. Q. Viet,  T. T. H. Thanh, {\it On $(FC)$-sequences and mixed multiplicities
of multigraded algebras}, Tokyo J. Math. 34 (2011), 185-202.
\bibitem{VT1} D. Q. Viet,  T. T. H. Thanh, {\it On  some  multiplicity and mixed multiplicity
formulas}, Forum Math. 26 (2014), 413-442.
\bibitem{VT2} D. Q. Viet,  T. T. H. Thanh, {\it A  note on formulas transmuting  mixed
multiplicities}, Forum Math. 26 (2014), 1837-1851.
\bibitem{VT3} D. Q. Viet,  T. T. H. Thanh, {\it The
Euler-Poincar\'{e} characteristic and  mixed multiplicities},
Kyushu J. Math. 69 (2015),  393-411.
\bibitem{VT4} D. Q. Viet,   T. T. H. Thanh,
{\it On the filter-regular sequences of multi-graded modules},
Tokyo J. Math. 38 (2015), 439-457.}
\end{thebibliography}
\end{document}